\documentclass{amsart}
\usepackage{amssymb}
\usepackage{graphicx}
\usepackage[shortlabels]{enumitem}
\usepackage{multirow}
\usepackage{amsmath,color}
\usepackage{hyperref}
\usepackage{url}

\DeclareMathOperator{\diag}{diag}

\newtheorem{theorem}{Theorem}[section]
\newtheorem{proposition}[theorem]{Proposition}
\newtheorem{lemma}[theorem]{Lemma}
\newtheorem{corollary}[theorem]{Corollary}

\newtheorem{remark}[theorem]{Remark}

\makeatletter
\@addtoreset{equation}{section}

\makeatother

\title[$\{0,\pm 1\}$-matrices with large determinants]{Symmetric and skew-symmetric $\{0,\pm 1\}$-matrices \\ with large determinants }
\author{ Gary Greaves }
\thanks{G.G. was supported by JSPS KAKENHI; 
grant number: 26$\cdot$03903.}
\address{School of Physical and Mathematical Sciences, 
Nanyang Technological University, 21 Nanyang Link, Singapore 637371}
\email{grwgrvs@gmail.com}

\author{Sho Suda}
\thanks{S.S. was supported by JSPS KAKENHI; grant number: 15K21075. }
\address{Department of Mathematics Education,  Aichi University of Education, Kariya, 448-8542, Japan}
\email{suda@auecc.aichi-edu.ac.jp}

\begin{document}



\begin{abstract}
	We show that the existence of $\{\pm 1\}$-matrices having largest possible determinant is equivalent to the existence of certain tournament matrices.
	In particular, we prove a recent conjecture of Armario.
	We also show that large submatrices of conference matrices are determined by their spectrum.
\end{abstract}

\maketitle



\section{Introduction}

Throughout, $I_n$, $J_n$, and $O_n$ will (respectively) always denote the $n \times n$ identity matrix, all-ones matrix, and all-zeros matrix.
We omit the subscript when the order is understood.
In Section~\ref{sec:submatrices_of_skew_hadamard_matrices}, the matrices denoted by $J$ and $O$ may not be square matrices.
We use $\mathbf{0}$ and $\mathbf{1}$ to denote the all-zeros and all-ones (column) vectors respectively.

Let $X$ be a $\{\pm 1\}$-matrix of order $n$.
As usual, $X$ is called \textbf{symmetric} if $X - X^\top = O$ and, abusing language, we call $X$ \textbf{skew-symmetric} if $X+X^\top = 2I$, i.e., the matrix $X-I$ is skew-symmetric in the usual sense.
We call $X$ a \textbf{D-optimal design} if the absolute determinant of $X$ is the maximum absolute determinant among all $\{\pm 1\}$-matrices of order $n$.
A famous inequality due to Hadamard~\cite{Hadamard1893} is the following:
\begin{equation}
	\label{ineq:Hadamard}
	|\det(X)| \leqslant n^{n/2}.
\end{equation}
Equality is achieved in \eqref{ineq:Hadamard} if and only if the columns of $X$ are orthogonal.
Furthermore if equality holds then $X$ is called a \textbf{Hadamard matrix}.
It is well-known that the order of a Hadamard matrix must be $1$, $2$, or a multiple of $4$ and it is conjectured that Hadamard matrices exist for all such orders.
Even stronger still, it is conjectured that skew-symmetric Hadamard matrices exist for all orders divisible by $4$.
There has been extensive work on Hadamard matrices, and we refer to the \emph{Handbook of Combinatorial Designs}~\cite{CRCHAndbook2007} for relevant background.

Hadamard's inequality can be improved if we restrict to matrices whose orders are not divisible by $4$.
Indeed, Ehlich~\cite{Ehlich1964} and Wojtas~\cite{Wojtas1964} independently showed that for a $\{\pm 1\}$-matrix $X$ of order $n \equiv 2\pmod{4}$, Hadamard's inequality can be strengthened to
\begin{align}\label{ineq:EW}
|\det(X)| \leqslant 2(n-1)(n-2)^{(n-2)/2}.
\end{align}
Moreover, there exists a $\{\pm 1\}$-matrix achieving equality in \eqref{ineq:EW} if and only if there exists a $\{\pm 1\}$-matrix $B$ such that
\begin{align}\label{eq:ew}
BB^\top=B^\top B=
\begin{pmatrix}
(n-2)I+2J & O_{n/2} \\
O_{n/2} & (n-2)I+2J
\end{pmatrix}.
\end{align}
A $\{\pm 1\}$-matrix of order $n$ is called an \textbf{EW matrix} if it satisfies \eqref{eq:ew}.
Clearly Hadamard matrices and EW matrices are D-optimal designs.
Note that it is known that EW matrices exist only if $2(n-1)$ is the sum of two squares, hence there is no analogue of the Hadamard conjecture for EW matrices.
We refer again to the \emph{Handbook of Combinatorial Designs}~\cite{CRCHAndbook2007} for background on EW matrices. 

A \textbf{tournament matrix} is a $\{0,1\}$-matrix $A$ such that $A + A^\top = J-I$. 
A tournament matrix is called \textbf{doubly regular} if, for some $t \in \mathbb N$, we have $AA^\top = tJ + (t+1)I$.
In 1972, Reid and Brown~\cite{ReidBrown72} showed that the existence of skew-symmetric Hadamard matrices of order $n$ is equivalent to the existence of doubly regular tournament matrices of order $n-1$.
Another important property of doubly regular tournaments is that they can be characterised by their spectrum \cite[Proposition 3.1]{DeCaen92}.

Recently, Armario~\cite[Page 10]{Armario15} conjectured that the existence of skew-symmetric EW matrices is equivalent to the existence of tournament matrices with a certain spectrum.
We write $\chi_M(x) := \det(xI - M)$ to denote the characteristic polynomial of a matrix $M$.
In this paper we prove Armario's conjecture, that is, we prove the following theorem.

\begin{theorem}\label{thm:Aconj}
	There exists a skew-symmetric EW matrix of order $4t+2$ if and only if there exists a tournament matrix $A$ with characteristic polynomial
    \[
      \chi_A(x) = \left ( x^3 - (2t-1)x^2 - t(4t-1) \right ) \left ( x^2+x+t \right )^{2t-1}.
    \]
\end{theorem}

A \textbf{Seidel matrix} is a $\{0,\pm 1\}$-matrix $S$ with zero diagonal and all off-diagonal entries nonzero such that $S = \pm S^\top$.
Note that our definition of a Seidel matrix is more general than the usual definition of a Seidel matrix, since we also allow skew-symmetric matrices.
A \textbf{conference matrix} is a Seidel matrix $S$ of order $n$ such that $SS^\top = (n-1)I$.
It is straightforward to deduce that the eigenvalues of a conference matrix are $\pm \sqrt{n-1}$ (resp. $\pm \sqrt{1-n}$) if it is symmetric (resp. skew-symmetric) with each of the two distinct eigenvalues having equal multiplicity.
The second half of the paper is devoted to studying large principal submatrices of conference matrices.
We show that, for all $t \in \mathbb N$, the existence of a skew-symmetric conference matrix of order $4t+4$ is equivalent to the existence of a skew-symmetric Seidel matrix of order $4t+i$ having a prescribed spectrum for each $i \in \{3,2,1\}$.
See Theorem~\ref{thm:skewMain}.
%
%
%
The existence of skew-symmetric Hadamard matrices is equivalent to the existence of skew-symmetric conference matrices.
Indeed, if $C$ is a skew-symmetric conference matrix then $C+I$ is a skew-symmetric Hadamard matrix.

We also establish a similar result for symmetric conference matrices, i.e., we show that, for all $t \in \mathbb N$, the existence of a symmetric conference matrix of order $4t+2$ is equivalent to the existence of a symmetric Seidel matrix of order $4t+i$ having a prescribed spectrum for each $i \in \{1,0,-1\}$.  
See Theorem~\ref{thm:symMain}.

%

Symmetric conference matrices of order $4t+2$ can be constructed from strongly regular graphs with parameters $(4t+1,2t,t-1,t)$.
(See Section 10.4 of Brouwer and Haemers' book~\cite{SpecGraphBH}.)

Now for the organisation of the paper.
In Section~\ref{sec:tournaments_and_skew_seidel_matrices}, we layout our main tools for the subsequents proofs.
We prove Theorem~\ref{thm:Aconj} in Section~\ref{sec:proof} and in Section~\ref{sec:submatrices_of_skew_hadamard_matrices} we show that large submatrices of conference matrices are determined by their spectra.

\section{Basic tools for Seidel matrices} 
\label{sec:tournaments_and_skew_seidel_matrices}

In this paper our main object of study are Seidel matrices.
In this section we state some basic results concerning these matrices.

Let $\mathcal O_n(\mathbb{Z})$ denote the orthogonal group generated by signed permutation matrices of order $n$.
We say that two $\mathbb{Z}$-matrices $A$ and $B$ (of order $n$) are \textbf{switching equivalent} if $A = P^{-1}BP$ for some matrix $P \in \mathcal O_n(\mathbb{Z})$.
It is clear that two switching-equivalent matrices have the same spectrum.
We call a matrix \textbf{normalised} if all nonzero entries of the first row equal $1$.
Note that every switching-equivalence class of Seidel matrices contains normalised matrices.

\subsection{Spectral tools} 
\label{sub:spectra}

Let $M$ be a matrix with $r$ distinct eigenvalues $\theta_1,\dots,\theta_r$ such that $\theta_i$ has (algebraic) multiplicity $m_i$.
We write the spectrum of $M$ as $\{ [\theta_1]^{m_1},\dots,[\theta_r]^{m_r} \}$ and we use $\Lambda(M)$ to denote the set $\{\theta_1,\dots,\theta_r\}$ of distinct eigenvalues of $M$.

Let $S$ be a Seidel matrix of order $n$.
Then $S$ is a normal matrix with diagonal entries $0$. 
Moreover, the diagonal entries of $S^2$ are all equal to $n-1$ (resp. $1-n$) if $S$ is symmetric (resp. skew-symmetric).
Putting the above facts about Seidel matrices together, we obtain the following.

\begin{proposition}\label{pro:basic}
	Let $S$ be a Seidel matrix of order $n$ with spectrum $\{ [\theta_1]^{m_1},\dots,[\theta_r]^{m_r} \}$.
	Then
	\begin{enumerate}[(i)]
		\item $\sum_{i=1}^r m_i = n$;
		\item $\sum_{i=1}^r m_i \theta_i = 0$;
		\item $\sum_{i=1}^r m_i \theta_i^2 = 
		\begin{cases}
			n(n-1), & \text{ if $S$ is symmetric} \\
			-n(n-1), & \text{ if $S$ is skew-symmetric.}
		\end{cases}
		$
	\end{enumerate}
\end{proposition}

Let $S$ be a skew-symmetric Seidel matrix.
Since the matrix $\sqrt{-1}S$ is Hermitean, all eigenvalues of such matrices are in $\sqrt{-1}\mathbb{R}$.
Furthermore, since $S$ is invariant under complex conjugation, the eigenvalues of $S$ must be symmetric about $0$.
Thus we have the following result.

\begin{proposition}\label{pro:skewProps}
	Let $S$ be a skew-symmetric Seidel matrix of order $n$.
	Then
	\begin{enumerate}[(i)]
		\item if $\theta \in \Lambda(S)$ then $-\theta \in \Lambda(S)$;
		\item if $n$ is odd then $0 \in \Lambda(S)$.
	\end{enumerate}
\end{proposition}

Since we do not have a symmetric version of Proposition~\ref{pro:skewProps}, we will use the following result about the product of eigenvalues of a symmetric Seidel matrix.

\begin{proposition}[Corollary 3.6 of \cite{Greaves2016208}]\label{pro:detS}
	Let $S$ be a symmetric Seidel matrix of order $n$.
	Then $\det S \equiv 1-n \pmod 4$.
\end{proposition}

We will also heavily use the famous interlacing theorem due to Cauchy~\cite{Cau:Interlace} (also see \cite{Fisk:Interlace05} for a short proof).
We will sometimes refer to the application of this theorem with the phrase `by interlacing'.

\begin{theorem}[Interlacing]\label{thm:interlacing}
	Let $H$ be an Hermitean matrix of order $n$ with eigenvalues $\lambda_1 \geqslant \dots \geqslant \lambda_n$.
	Let $H^\prime$ be a principal submatrix of $H$ of order $m$ with eigenvalues $\mu_1 \geqslant \dots \geqslant \mu_m$.
	Then $\lambda_i \geqslant \mu_i \geqslant \lambda_{n-m+i}$ for all $i \in \{1,\dots,m\}$.
\end{theorem}

%


\subsection{Main angles} 
\label{sub:main_angles}

	Let $N$ be a normal matrix of order $n$.
	For each eigenvalue $\theta$ of $N$, let $P_\theta$ be the orthogonal projection onto the eigenspace corresponding to $\theta$.
	Then $N$ has spectral decomposition
	\begin{align}
		N &= \sum_{\theta \in \Lambda(N)} \theta P_\theta, \text{ and } \label{eqn:MA1} \\
		I &= \sum_{\theta \in \Lambda(N)} P_\theta. \label{eqn:MA2}
	\end{align}
	The \textbf{main angles} of $N$ are defined as $\beta_\theta := \frac{1}{\sqrt{n}} || P_\theta \mathbf{1}||$ for each $\theta \in \Lambda(N)$.
	Main angles were originally defined to study the eigenspaces of graphs~\cite{Cvet:1971} (see \cite{Rowlinson:2007} for a survey).
	Let $\theta$ be an eigenvalue of $N$.
	Define $\alpha_\theta$ by
	\[
	\alpha_\theta:= ||P_\theta \mathbf{1} ||^2.
	\]
	Note that $\alpha_\theta = n\beta_\theta^2$.
	We use this alternative definition for notational convenience.
	
	We record some properties for the $\alpha_\theta$ in the following proposition.
	\begin{proposition}\label{pro:meinangles}
		Let $S$ be a skew-symmetric Seidel matrix of order $n$.
		Then for all $\theta \in \Lambda(S)$, we have $\alpha_{\theta} = \alpha_{-\theta}$.
		Furthermore
		\begin{align}
			\sum_{\theta \in \Lambda(S)} \theta^2 \alpha_\theta &= \mathbf{1}^\top S^2 \mathbf{1}; \label{eqn:sum21} \\
			\sum_{\theta \in \Lambda(S)} \alpha_\theta &= n. \nonumber 
		\end{align}
	\end{proposition}
	\begin{proof}
		Let $\theta$ be an eigenvalue of $S$.
		If $\mathbf{v}$ is an eigenvector for $\theta$ then its complex conjugate is an eigenvector for $-\theta$.
		Hence $\alpha_{\theta} = \alpha_{-\theta}$ for all $\theta \in \Lambda(S)$.
		The other two equalities follow from Equations~\eqref{eqn:MA1} and \eqref{eqn:MA2}, using the fact that the $P_{\theta}$ are mutually orthogonal idempotents.
	\end{proof}
	
	The next lemma gives a sufficient condition for eigenvalues of a tournament matrix to have main angle equal to zero.
	
	\begin{lemma}[Lemma 1 of \cite{KirkShad:1994}]\label{lem:re12}
		Let $A$ be a tournament matrix and let $\mathbf{v}$ be an eigenvector of $A$ corresponding to an eigenvalue $\theta$ with $\operatorname{Re}(\theta) = -1/2$.
		Then $\mathbf{v}^\top \mathbf{1} = 0$.
	\end{lemma}
	
	The utility of main angles is to obtain expressions for the characteristic polynomial of matrices that differ by some multiple of $J$.
	The next result follows from \cite[Corollary 2.4]{NozSuda:12}.
	
	\begin{lemma}\label{lem:charconvert}
		Let $A$ be a matrix such that $S = J-I-2A$ is a normal matrix and let $c$ be an indeterminate.
		Then
		\begin{align}
			\chi_{S+cJ}(x) &= \chi_S(x)\left (1 - c\sum_{\theta \in \Lambda(S)} \frac{\alpha_\theta}{x-\theta} \right); \label{eqn:conv1} \\
			\chi_{A}(x) &= \left (\frac{-1}{2} \right )^n \chi_{S-J}(-2x-1). \label{eqn:conv2}
		\end{align}
	\end{lemma}

%
%
%


\section{Proof of Armario's conjecture} 
\label{sec:proof}
	
	In this section we prove Theorem~\ref{thm:Aconj}.
	We begin with a result about the entries of $A \mathbf{1}$, where $A$ is a tournament matrix associated to a normalised skew-symmetric EW matrix.
	
	\begin{lemma}[{\cite[Lemma 1]{Armario15}}]\label{lem:scorevec}
		Let $M$ be a normalised skew-symmetric EW matrix of order $4t + 2$ and let $A$ be the matrix obtained by deleting the first row and column of $(J-M)/2$.
		Then the entries of $A \mathbf{1}$ are $2t+1$, $2t$, and $2t-1$, each appearing $t$, $2t+1$, and $t$ times respectively.
	\end{lemma}
	
	We will work directly with the vector $S \mathbf{1}$ where $S$ is a normalised skew-symmetric EW matrix.
	This follows straightforwardly from the above result.
	
	\begin{corollary}\label{cor:Sscore}
		Let $M$ be a normalised skew-symmetric EW matrix of order $4t + 2$ and let $S = M-I$.
		Then the entries of $S \mathbf{1}$ are $4t+1$, $1$, $-1$, and $-3$ each appearing $1$, $t$, $2t+1$, and $t$ times respectively.
	\end{corollary}

	Next we write expressions for the eigenvalues and main angles of a normalised skew-symmetric EW matrix in terms of its order. 

	\begin{lemma}\label{lem:mainangs}
		Let $M$ be a normalised skew-symmetric EW matrix of order $4t + 2$.
		Then $M-I$ has spectrum $\{ [\pm \lambda]^1, [\pm \mu]^{2t} \}$, where $\lambda = \sqrt{-8t-1}$ and $\mu = \sqrt{1-4t}$.
		Moreover the eigenvalues $\pm \lambda$ each have $\alpha_{\pm \lambda} = (4t+1)/(2t+1)$ and $\pm \mu$ each have $\alpha_{\pm\mu} = 2t/(2t+1)$.
	\end{lemma}
	\begin{proof}
		The spectrum is straightforward to calculate using Equation~\eqref{eq:ew}.

		Using Corollary~\ref{cor:Sscore} and Proposition~\ref{pro:meinangles}, we find that $\alpha_{\lambda}$ and $\alpha_{\mu}$ satisfy the simultaneous equations
		\begin{align*}
			2\alpha_{\lambda} + 4t\alpha_{\mu} &= 4t+2 \\
			2 \lambda^2 \alpha_{\lambda} + 4t \mu^2 \alpha_{\mu} &= (4t+1)^2 + (2t+1) + t + 9t.
		\end{align*}
		From which the lemma follows.
	\end{proof}
	
	Now we can show one direction of Theorem~\ref{thm:Aconj}.

	\begin{lemma}\label{lem:armario}
		Let $M$ be a skew-symmetric EW matrix of order $4t + 2$.
		Then there exists a tournament matrix $A$ of order $4t+1$ with characteristic polynomial
		\[
			\chi_A(x) = \left ( x^3 - (2t-1)x^2 - t(4t-1) \right ) \left ( x^2+x+t \right )^{2t-1}.
		\]
	\end{lemma}
	\begin{proof}
		Let $S = M - I$ and assume that $S$ is normalised.
		By Lemma~\ref{lem:mainangs} the matrix $S$ has spectrum $\{ [\pm \lambda]^1, [\pm \mu]^{2t} \}$, where $\lambda = \sqrt{-8t-1}$ and $\mu = \sqrt{1-4t}$ and main angles $\alpha_{\pm \lambda} = (4t+1)/(2t+1)$ and $\alpha_{\pm\mu} = 2t/(2t+1)$.
		Let $A^\prime$ be the tournament matrix with Seidel matrix $S$.
		Using Lemma~\ref{lem:charconvert}, we see that $A^\prime$ has characteristic polynomial
		\[
			\chi_{A^\prime}(x) = x \left ( x^3 - (2t-1)x^2 - t(4t-1) \right ) \left ( x^2+x+t \right )^{2t-1}.
		\]
		Since $S$ is normalised, the tournament matrix $A^\prime$ has the form
		\[
			A^\prime =
			\begin{pmatrix}
				0  & \mathbf{0}^\top \\
				\mathbf{1} & A
			\end{pmatrix},
		\]
		where $A$ is also a tournament matrix.
		It is easy to see that $\chi_{A^\prime}(x) = x \chi_A(x)$ as required.
	\end{proof}
	
	It remains to show the other direction of Theorem~\ref{thm:Aconj}.
	First we state a result about the structure of a positive semidefinite matrix with constant diagonal $1$.
	
	\begin{lemma}[Lemma 5.24 of \cite{Greaves2016208}]\label{lem:posSemi}
		Let $M$ be a positive semidefinite $\{0,\pm 1\}$-matrix with constant diagonal entries $1$.
		Then there exist positive integers $c$, $k_1, \dots, k_c$, such that $M$ is switching equivalent to the block diagonal matrix $\diag(J_{k_1},\dots,J_{k_c})$.
	\end{lemma}
	
	Next we state a useful lemma about positive semidefinite matrices with constant diagonals.
	
	\begin{lemma}\label{lem:posDef}
		Let $S$ be a Seidel matrix of order $n$ such that $M:=SS^\top - aI$ is positive semidefinite for some integer $a$.
		Then each off-diagonal entry of $M$ is congruent to $n$ modulo $2$ and has absolute value at most $n-1-a$.
	\end{lemma}
	\begin{proof}
		The diagonal entries of $M$ are each equal to $n-1-a$.
		Since $M$ is positive semidefinite, by interlacing, each $2 \times 2$ principal submatrix is positive semidefinite.
		Hence each entry of $M$ has absolute value at most $n-1-a$.
		Furthermore, the inner product of any distinct two rows of $S$ has parity equal to the parity of $n$.
	\end{proof}
	
	Second we show how to go from a skew-symmetric Seidel matrix with a certain characteristic polynomial to a skew-symmetric EW matrix.
	
	\begin{lemma}\label{lem:switchingForm}
		Let $S$ be a skew-symmetric Seidel matrix with characteristic polynomial $\chi_S(x) = (x^2 + 8t+1)(x^2+4t-1)^{2t}$.
		Then $S$ is switching equivalent to a skew-symmetric Seidel matrix $T$ such that $T + I$ is a skew-symmetric EW matrix.
	\end{lemma}
	\begin{proof}
		Set $M := SS^\top - (4t-1)I$.
		Then $M$ has spectrum $\{ [0]^{4t}, [4t+2]^2\}$.
		By Lemma~\ref{lem:posDef} all entries of $M$ belong to the set $\{0,\pm 2\}$.
		Hence, using Lemma~\ref{lem:posSemi}, we can deduce that $M$ is switching equivalent to the block-diagonal matrix $\diag(2J,\dots, 2J)$.
		Since the rank of $M$ is $2$ and the nonzero eigenvalues of $M$ are both $4t+2$, we have that $M$ is switching equivalent to $\diag(2J_{2t+1},2J_{2t+1})$, as required.
	\end{proof}
	
	Finally we prove the other direction of Theorem~\ref{thm:Aconj}.

  \begin{lemma}\label{lem:armarioCon}
     Let $A$ be a tournament matrix of order $4t+1$ with characteristic polynomial
    \[
      \chi_A(x) = \left ( x^3 - (2t-1)x^2 - t(4t-1) \right ) \left ( x^2+x+t \right )^{2t-1}.
    \]
    Then there exists a skew-symmetric EW matrix of order $4t + 2$.
  \end{lemma}
  \begin{proof}
  	Let $S^\prime = J - I -2A$.
	By Lemma~\ref{lem:re12}, the eigenspace of the eigenvalues of $A$ equal to $(-1\pm \sqrt{1-4t})/2$ is also the eigenspace of eigenvalues of $S^\prime$ equal to $\mp \mu$, where $\mu = \sqrt{1-4t}$.
	Since $S^\prime$ is a skew-symmetric matrix of odd order, we know that it must have a zero eigenvalue.
	We have thus accounted for all but two (including multiplicity) of the eigenvalues of $S^\prime$. 
	
	Let $\theta_1$ and $\theta_2$ denote to two remaining unknown eigenvalues of $S^\prime$.
	Using Proposition~\ref{pro:basic}, we have that $\{ \theta_1, \theta_2\} = \{ \pm \lambda \}$, where $\lambda = \sqrt{1-8t}$.
	Hence $S^\prime$ has spectrum $\{ [\pm \lambda]^1, [\pm \mu]^{2t-1}, [0]^1 \}$.
	
	Now for the main angles.
	From above we know that $\alpha_{\pm \mu} = 0$.
	By Proposition~\ref{pro:meinangles} we have $\sum_{\theta \in \Lambda(S)} \alpha_\theta = 4t+1$ and $\alpha_{\lambda} = \alpha_{-\lambda}$.
	Using Equations~\eqref{eqn:conv1} and~\eqref{eqn:conv2}, we can consider equality for the constant term of the polynomials.
	Together with Equation~\eqref{eqn:sum21}, we obtain the simultaneous equations
	\begin{align*}
		2\alpha_{\lambda} + \alpha_{0} &= 4t+1 \\
		(8t-1)(4t\alpha_0 - \alpha_\lambda) &= 4t((8t+1)(4t-1)-1).
	\end{align*}
	Thus $\alpha_\lambda = (8t+1)(4t-1)/(8t-1)$ and $\alpha_0 = 4t/(8t-1)$.
	
	Now define $S$ as
	\[
		S = 
		\begin{pmatrix}
		0  & \boldsymbol{1}^\top \\
		-\boldsymbol{1} & S^\prime
		\end{pmatrix}.
	\]
    Then we have
    \begin{align*}
    	\chi_S(x) &= \det \begin{pmatrix}
  		x  & -\boldsymbol{1}^\top \\
  		\boldsymbol{1} & xI - S^\prime
  		\end{pmatrix} \\
  		&=  \det \begin{pmatrix}
  		x  & -\boldsymbol{1}^\top \\
  		\mathbf{0}& xI - S^\prime + \frac{1}{x}J
  		\end{pmatrix} \\
  		&= x\chi_{S^\prime - \frac{1}{x}J}(x).
    \end{align*}
    Therefore, by Lemma~\ref{lem:charconvert}, we have $\chi_S(x) = (x^2 + 8t+1)(x^2+4t-1)^{2t}$.
	
	By Lemma~\ref{lem:switchingForm}, $S$ is switching equivalent to a skew-symmetric Seidel matrix $T$ such that $T+I$ is a skew-symmetric EW matrix of order $4t+2$.
  \end{proof}
 

\section{Submatrices of conference matrices} 
\label{sec:submatrices_of_skew_hadamard_matrices}

In this section we show that the existence of Seidel matrices that are cospectral with large principal submatrices of conference matrices is equivalent to the existence of conferences matrices themselves.
Furthermore, we show how to construct the corresponding conference matrices from these smaller order matrices.

We state our results separately for the skew-symmetric and symmetric cases.
We have written proofs for the skew-symmetric case since the results for the symmetric case essentially follow \emph{mutatis mutandis}.
See Remark~\ref{rem:extra} for where some extra work is required.

\begin{theorem}\label{thm:skewMain}
	The existence of the following are equivalent:
	\begin{enumerate}[(i)]
		\item a skew-symmetric Seidel matrix with characteristic polynomial $$(x^2+4t+3)^{2t+2};$$
		\item a skew-symmetric Seidel matrix with characteristic polynomial $$x(x^2+4t+3)^{2t+1};$$
		\item a skew-symmetric Seidel matrix with characteristic polynomial $$(x^2 + 1)(x^2+4t+3)^{2t};$$
		\item a skew-symmetric Seidel matrix with characteristic polynomial $$x(x^2 + 3)(x^2+4t+3)^{2t-1}.$$
	\end{enumerate}
\end{theorem}

\begin{theorem}\label{thm:squareFormSkew}
	Let $S$ be a Seidel matrix.
	Then 
	\begin{enumerate}[(i)]
		\item $\chi_S(x) = (x^2+4t+3)^{2t+2}$ if and only if $(4t+3)I + S^2 = O$;

		\item $\chi_S(x) = x(x^2+4t+3)^{2t+1}$ if and only if $(4t+3)I + S^2$ is switching equivalent to $J_{4t+3}$;

		\item $\chi_S(x) = (x^2 + 1)(x^2+4t+3)^{2t}$ if and only if $(4t+3)I + S^2$ is switching equivalent to
		\[
			\begin{pmatrix}
						  				2J_{2t+1} & O \\
						  				O & 2J_{2t+1}
			\end{pmatrix};
		\]
		\item $\chi_S(x) = x(x^2 + 3)(x^2+4t+3)^{2t-1}$ if and only if $(4t+3)I + S^2$ is switching equivalent to
		\[
			\begin{pmatrix}
						  				3J_{t+1} & J & J & J \\
						  				J & 3 J_t & -J & -J \\
						  				J & -J & 3J_t & -J \\
						  				J & -J & -J & 3J_t
			\end{pmatrix}.
		\]
	\end{enumerate}
\end{theorem}

We have similar statements for symmetric conference matrices and their principal submatrices.

\begin{theorem}\label{thm:symMain}
	The existence of the following are equivalent:
	\begin{enumerate}[(a)]
		\item a symmetric Seidel matrix with characteristic polynomial $$(x^2-4t-1)^{2t+1};$$
		\item a symmetric Seidel matrix with characteristic polynomial $$x(x^2-4t-1)^{2t};$$
		\item a symmetric Seidel matrix with characteristic polynomial $$(x^2 - 1)(x^2-4t-1)^{2t-1};$$
		\item a symmetric Seidel matrix with characteristic polynomial $$(x - 2)(x + 1)^2(x^2-4t-1)^{2(t-1)}.$$
	\end{enumerate}
\end{theorem}

\begin{remark}
	\label{rem:extra}
	To prove that the existence of a symmetric Seidel matrix with characteristic polynomial $(x^2 - 1)(x^2-4t-1)^{2t-1}$ implies the existence of a symmetric Seidel matrix with characteristic polynomial $(x\pm 2)(x\mp 1)^2(x^2-4t-1)^{2(t-1)}$ requires a bit more work than in the skew-symmetric case.
	Hence we give an argument here.
	
	Let $S$ be a Seidel matrix satisfying (c) of Theorem~\ref{thm:symMain}.
	Then $S$ has spectrum $\{ [\pm \sqrt{4t+1}]^{2t-1}, [\pm 1]^1 \}$.
	Let $S^\prime$ be a principal submatrix of order $4t-1$.
	By interlacing, $S^\prime$ has eigenvalues $\pm \sqrt{4t+1}$ each with multiplicity at least $2(t-1)$.
	Let $\alpha$, $\beta$, and $\gamma$ denote the remaining three unknown eigenvalues.
	Using Proposition~\ref{pro:basic}, we have that $\alpha+\beta+\gamma = 0$ and $\alpha^2+\beta^2+\gamma^2 = 6$.
	By a result of Karapiperi et al.~\cite[Proposition 4]{karapiperi2012eigenvalue}, the determinant of $S^\prime$ is equal to $2(4t+1)^{2(t-1)}$.
	Hence we have $\alpha\beta\gamma = 2$.
	Together with the above equations for the power-sums, we find that $\left \{ \alpha, \beta, \gamma \right \} = \left \{ [2]^1, [-1]^{2} \right \}$.
\end{remark}

\begin{theorem}\label{thm:squareFormSym}
	Let $S$ be a Seidel matrix.
	Then 
	\begin{enumerate}[(a)]
		\item $\chi_S(x) = (x^2-4t-1)^{2t+1}$ if and only if $(4t+1)I - S^2 = O$;
		\item $\chi_S(x) = x(x^2-4t-1)^{2t}$ if and only if $(4t+1)I - S^2$ is switching equivalent to $J_{4t+1}$;
		\item $\chi_S(x) = (x^2 - 1)(x^2-4t-1)^{2t-1}$ if and only if $(4t+1)I - S^2$ is switching equivalent to
		\[
			\begin{pmatrix}
						  				2J_{2t} & O \\
						  				O & 2J_{2t}
			\end{pmatrix};
		\]
		\item $\chi_S(x) = (x\pm 2)(x\mp 1)^2(x^2-4t-1)^{2(t-1)}$ if and only if $(4t+1)I - S^2$ is switching equivalent to
		\[
			\begin{pmatrix}
						  				3J_{t-1} & J & J & J \\
						  				J & 3 J_t & -J & -J \\
						  				J & -J & 3J_t & -J \\
						  				J & -J & -J & 3J_t
			\end{pmatrix}.
		\]
	\end{enumerate}
\end{theorem}

\subsection{Proof for the skew-symmetric case} 
\label{sub:proof_for_the_skew_symmetric_case}

	Now we prove Theorem~\ref{thm:skewMain} and Theorem~\ref{thm:squareFormSkew}.
	First we deal with the easier directions of both of these theorems.
	
	For Theorem~\ref{thm:squareFormSkew}, it is clear that matrices whose squares have the forms stated in Theorem~\ref{thm:squareFormSkew} have the corresponding spectra.
	
	For Theorem~\ref{thm:skewMain}, suppose that $S$ is a skew-symmetric Seidel matrix with characteristic polynomial $(x^2 + 1)(x^2+4t+3)^{2t}$.
	Then $\sqrt{-1}S$ is an Hermitean matrix with spectrum $\{ [\pm \sqrt{4t+3}]^{2t}, [\pm 1]^1 \}$.
	Let $S^\prime$ be a principal submatrix of $\sqrt{-1}S$ of order $4t+1$.
	By interlacing (Theorem~\ref{thm:interlacing}), $S^\prime$ has eigenvalues $\pm \sqrt{4t+3}$ each with multiplicity at least $2t-1$.
	Hence $\sqrt{-1}S^\prime$ is a principal submatrix of $-S$ and $\sqrt{-1}S^\prime$ has eigenvalues $\pm \sqrt{-4t-3}$ each with multiplicity at least $2t-1$.
	By Proposition~\ref{pro:skewProps}, since $4t+1$ is odd, $\sqrt{-1}S^\prime$ must have a zero eigenvalue.
	The remaining two unknown eigenvalues of $\sqrt{-1}S^\prime$ can be determined using Proposition~\ref{pro:basic}.
	
	This shows that the existence of a matrix satisfying (iii) of Theorem~\ref{thm:skewMain} implies the existence of a matrix satisfying (iv) of Theorem~\ref{thm:skewMain}.
	The other statements of Theorem~\ref{thm:skewMain} in this direction are easier and follow in a similar way.
	
	We note that, in the symmetric case, to show that the existence of a matrix satisfying (c) of Theorem~\ref{thm:symMain} implies the existence of a matrix satisfying (d) of Theorem~\ref{thm:symMain}, one also uses Proposition~\ref{pro:detS} (see Remark~\ref{rem:extra}).
	This extra condition is needed since we do not have a symmetric version of Proposition~\ref{pro:skewProps}.
	
	Now we begin a series of lemmata which will constitute the proofs for Theorem~\ref{thm:skewMain} and Theorem~\ref{thm:squareFormSkew}.
	The first lemma gives a proof for (ii) of Theorem~\ref{thm:squareFormSkew}.
	
	\begin{lemma}\label{lem:del1}
		Let $S$ be a skew-symmetric Seidel matrix of order $4t+3$ and let $M=(4t+3)I + S^2$.
		Then $S$ has spectrum 
		\[
			\{ [\pm \sqrt{-4t-3}]^{2t+1}, [0]^1\}
		\]
		if and only if $M$ is switching equivalent to $J$.
	\end{lemma}
	\begin{proof}
		Suppose $S$ has the assumed spectrum.
		Then $M$ has spectrum $\{ [0]^{4t+2}, [4t+3]^1\}$.
		By Lemma~\ref{lem:posDef}, the entries of $M$ are members of the set $\{ \pm 1 \}$.
		Now, apply Lemma~\ref{lem:posSemi} to deduce that $M$ has the form $\diag(J,\dots,J)$.
		And since the rank of $M$ is $1$, we have that $M$ has the desired form.
	\end{proof}
	
	Using the above lemma we can prove that the existence of matrices satisfying (ii) of Theorem~\ref{thm:skewMain} implies the existence of matrices satisfying (i) of Theorem~\ref{thm:skewMain}.
	
	\begin{lemma}\label{lem:del1up}
		Let $S$ be a skew-symmetric Seidel matrix of order $4t+3$ with spectrum 
		\[
			\{ [\pm \sqrt{-4t-3}]^{2t+1}, [0]^1\}.
		\]
		Then there exists a vector $\mathbf{x} \in \{\pm 1\}^{4t+3}$ such that the matrix
		\[
			\begin{pmatrix}
				S & \mathbf{x} \\
				-\mathbf{x}^\top & 0
			\end{pmatrix}
		\]
		is a skew-symmetric Seidel matrix with spectrum $\{ [\pm \sqrt{-4t-3}]^{2t+2} \}$.
	\end{lemma}
	\begin{proof}
		Let $M=(4t+3)I + S^2$.
		By Lemma~\ref{lem:del1}, we can assume that $M = J$.
		It is then easy to see that the nullspace of $S$ is spanned by $\mathbf{1}$.
		Now consider the skew-symmetric matrix 
		\[
			T = 
			\begin{pmatrix}
				S & \mathbf{1} \\
				-\mathbf{1}^{\top} & 0
			\end{pmatrix}.
		\]
		Then $T^2 = -(4t+3)I$ as required.
	\end{proof}
	
	The next lemma gives a proof for (iii) of Theorem~\ref{thm:squareFormSkew}.

	\begin{lemma}\label{lem:del2}
		Let $S$ be a skew-symmetric Seidel matrix of order $4t+2$ and let $M=(4t+3)I + S^2$.
		Then $S$ has spectrum 
		\[
			\{ [\pm \sqrt{-4t-3}]^{2t}, [\pm \sqrt{-1}]^1\}
		\]
		if and only if $M$ is switching equivalent to
		\[
			  \begin{pmatrix}
				2J_{2t+1} & O \\
				O & 2J_{2t+1}
			\end{pmatrix}.
		\]
	\end{lemma}
	\begin{proof}
		Suppose $S$ has the assumed spectrum.
		Then $M$ has spectrum $\{ [0]^{4t}, [4t+2]^2\}$.
		By Lemma~\ref{lem:posDef}, entries of $M$ are members of the set $\{ 0, \pm 2 \}$.
		Now, apply Lemma~\ref{lem:posSemi} to the matrix $M/2$, to deduce that $M/2$ has the form $\diag(J,\dots,J)$.
		Since the rank of $M$ is $2$, we have that $M$ has the desired form.
	\end{proof}
	
	Just like with Lemma~\ref{lem:del1up}, we use Lemma~\ref{lem:del2} to prove that the existence of matrices satisfying (iii) of Theorem~\ref{thm:skewMain} implies the existence of matrices satisfying (ii) of Theorem~\ref{thm:skewMain}.
	
	\begin{lemma}\label{lem:del2up}
		Let $S$ be a skew-symmetric Seidel matrix of order $4t+2$ with spectrum 
		\[
			\{ [\pm \sqrt{-4t-3}]^{2t}, [\pm \sqrt{-1}]^1\}.
		\]
		Then there exists a vector $\mathbf{x} \in \{\pm 1\}^{4t+2}$ such that the matrix
		\[
			\begin{pmatrix}
				S & \mathbf{x} \\
				-\mathbf{x}^\top & 0
			\end{pmatrix}
		\]
		is a skew-symmetric Seidel matrix with spectrum $\{ [\pm \sqrt{-4t-3}]^{2t+1}, [0]^1 \}$.
	\end{lemma}
	\begin{proof}
		Let $M=(4t+3)I + S^2$.
		By Lemma~\ref{lem:del2}, we can assume that
		\[
			M = \begin{pmatrix}
				2J_{2t+1} & O \\
				O & 2J_{2t+1}
			\end{pmatrix}.
		\]
		By inspection, it is evident that the eigenspace $\mathcal E$ of $M$ corresponding to the eigenvalue $4t+2$ is spanned by the vectors $\mathbf{1}$ and $\mathbf{1}_* := ( \mathbf{1}^\top, -\mathbf{1}^\top)^\top$.
		
		Let $\mathbf{v}$ be an eigenvector of $S$ for the eigenvalue $\sqrt{-1}$.
		Then $\overline {\mathbf{v}}$ is an eigenvector of $S$ for the eigenvalue $-\sqrt{-1}$.
		The eigenspace $\mathcal E$ is also spanned by $\mathbf{v}$ and $\overline {\mathbf{v}}$.
		Hence, we can choose $\mathbf{v}$ such that  $\mathbf{1} = \operatorname{Re}(\mathbf{v})$ and $\mathbf{1}_* = \operatorname{Im}(\mathbf{v})$.
		Therefore $S\mathbf{1}_*= \mathbf{1}$ and $\mathbf{1}_*^{\top} S= \mathbf{1}^\top$.
		
		Now consider the skew-symmetric matrix 
		\[
			T = 
			\begin{pmatrix}
				S & \mathbf{1}_* \\
				-\mathbf{1}_*^{\top} & 0
			\end{pmatrix}.
		\]
		Its square has the form
		\[
			T^2 = 
			\begin{pmatrix}
				S^2 - \mathbf{1}_*\mathbf{1}_*^{\top} & S\mathbf{1}_* \\
				-\mathbf{1}_*^{\top} S & -\mathbf{1}_*^{\top}\mathbf{1}_*
			\end{pmatrix}
			=
			\begin{pmatrix}
				-(4t+3)I+J & \mathbf{1} \\
				\mathbf{1}^\top & -(4t+2)
			\end{pmatrix}
			=
			-(4t+3)I + J.
		\]
		 It is therefore evident that $T$ has spectrum $\{ [\pm \sqrt{-4t-3}]^{2t+1}, [0]^1\}$.
	\end{proof}
%
	
	At this point it remains to prove (iv) of Theorem~\ref{thm:squareFormSkew} and that the existence of matrices satisfying (iv) in Theorem~\ref{thm:skewMain} implies the existence of matrices satisfying (iii).
	To do this we need a structural result for positive semidefinite matrices with constant diagonal $3$.

	\begin{lemma}\label{lem:33psd}
		Up to switching equivalence there are precisely four positive semidefinite $3 \times 3$ $\{\pm 1, \pm 3\}$-matrices with diagonal entries equal to $3$ given as follows:
		\[
			\begin{pmatrix}
				3 & 1 & 1 \\
				1 & 3 & 1 \\
				1 & 1 & 3
			\end{pmatrix}, 
			\begin{pmatrix}
				3 & 1 & 1 \\
				1 & 3 & -1 \\
				1 & -1 & 3
			\end{pmatrix},
			\begin{pmatrix}
				3 & 3 & 1 \\
				3 & 3 & 1 \\
				1 & 1 & 3
			\end{pmatrix},
			\text{ and }
			\begin{pmatrix}
				3 & 3 & 3 \\
				3 & 3 & 3 \\
				3 & 3 & 3
			\end{pmatrix}.
		\]
	\end{lemma}
	\begin{proof}
		This can be checked exhaustively by hand or computer.
	\end{proof}
	
	Using Lemma~\ref{lem:33psd}, we obtain the following structural result for positive semidefinite $\{\pm 1, \pm 3\}$-matrices with constant diagonal $3$.
	
	\begin{lemma}\label{lem:PDform}
		Let $M$ be a positive semidefinite $\{\pm 1, \pm 3\}$-matrix with diagonal entries equal to $3$.
		Suppose that the first row $\mathbf{r}_1$ has its first $a \geqslant 1$ entries equal to $3$.
		Then there exists a matrix $P \in \mathcal O_n(\mathbb Z)$ such that the first $a$ rows of $P^\top MP$ are equal to $\mathbf{r}_1$.
	\end{lemma}
	\begin{proof}
		The lemma is trivial for $n \leqslant 2$ and for $n = 3$ it follows from Lemma~\ref{lem:33psd}.
		Fix $n \geqslant 4$. 
		The case $a = 1$ is vacuous.
		First we deal with the case $a = 2$.
		For any $3 \leqslant i \leqslant n$ we can consider the $3 \times 3$ principal submatrix induced by the rows (and columns) $1$, $2$, and $i$.
		This matrix $X$ has the form
		\[
			\begin{pmatrix}
				3 & 3 & x \\
				3 & 3 & y \\
				x & y & 3
			\end{pmatrix}.
		\]
		Since the property of being positive semidefinite is preserved for principal submatrices, $X$ must be switching equivalent to one of the matrices in Lemma~\ref{lem:33psd}.
		Hence we have $x = y$ for all $i$.
		
		We assume that the lemma is true for all $a$ less than some $k \geqslant 3$.
		Consider the case $a = k$.
		By our assumption, for all $i, j < k$, we can assume (up to switching equivalence) that $M_{i,j} = 3$.
		We also know that $M_{1,k} = 3$ and $M_{k,k} = 3$.
		
		Now, for any $1 < j < k$, we can form a $3 \times 3$ principal submatrix of $M$ induced by the rows (and columns) $1$, $j$, and $k$.
		This matrix $X$ has the form
		\[
			\begin{pmatrix}
				3 & 3 & 3 \\
				3 & 3 & y \\
				3 & y & 3
			\end{pmatrix}.
		\]
		Again, this matrix must be switching equivalent to one of the matrices in Lemma~\ref{lem:33psd}, hence $y = 3$ for all $1 < j < k$.
	\end{proof}
	
	
	Now we can prove (iv) of Theorem~\ref{thm:squareFormSkew}.

	\begin{lemma}\label{lem:bfofM}
		Let $S$ be a skew-symmetric Seidel matrix of order $4t+1$ and let $M=(4t+3)I + S^2$.
		Then $S$ has spectrum 
		\[
			\{ [\pm \sqrt{-4t-3}]^{2t-1}, [\pm \sqrt{-3}]^1, [0]^1\}
		\]
		if and only if there exists some $P \in \mathcal O_{4t+1}(\mathbb{Z})$ such that
		\begin{equation}
		  P^\top M P = \begin{pmatrix}
			  				3J_{t+1} & J & J & J \\
			  				J & 3 J_t & -J & -J \\
			  				J & -J & 3J_t & -J \\
			  				J & -J & -J & 3J_t
			  			\end{pmatrix}. \label{eqn:sqForm3}
		\end{equation}
		Moreover, the first row of $P^\top M P$ is a null vector of $P^\top S P$.
	\end{lemma}
	\begin{proof}
%
		Let $M^\prime = 3I + S^2$.
		Then $MM^\prime = S^4 + (4t+6)S^2 + 3(4t+3)I = \mathbf{v}\mathbf{v}^\top$, where $\mathbf{v}$ spans the $0$-eigenspace of $S$.
		Whence $\mathbf{v}^\top \mathbf{v} = 3(4t+1)$.
		Consider the $(1,1)$-th entry $e_{11}$ of $M M^\prime$. 
		By above, we see that $e_{11} = v_1^2 \geqslant 0$.
		
		The matrix $M$ has spectrum $\{ [0]^{4t-2}, [4t]^2, [4t+3]^1\}$.
             By Lemma~\ref{lem:posDef}, entries of $M$ are members of the set $\{\pm 1, \pm 3\}$.
		Thus, by switching (if necessary), we can assume that the first row of $M$ consists only of entries equal to $1$ or $3$.
		Suppose the first row of $M$, $\mathbf{r}_1$ (say), has its first $a$ entries equal to $3$ and the remaining $4t+1-a$ entries equal to $1$. 
		Then $e_{11} = 8a -2(4t+3) + 7$.
		Hence, from the inequality $e_{11} \geqslant 0$, we have the lower bound for $a$,
		\begin{equation}
			\label{eqn:LB4a}
			a \geqslant \frac{8t-1}{8}.
		\end{equation}
		For an upper bound we use Cauchy-Schwarz for $\mathbf{v}$ and $\mathbf{r}_1$.
		Since $\mathbf{v}$ is in the $(4t+3)$-eigenspace of $M$, $\mathbf{r}_1^\top \mathbf{v} = (4t+3) v_1$.
		From above, observe that $e_{11} = v_1^2 = 8a -2(4t+3) +7$. 
		Hence we have
		\begin{equation}
			\label{eqn:UB4a}
			a \leqslant t+1.
		\end{equation}
	
		Using the upper and lower bounds \eqref{eqn:LB4a} and \eqref{eqn:UB4a}, we deduce that $a = t$ or $a = t+1$.
		
		By Lemma~\ref{lem:PDform}, we can assume the first $a$ rows of $S$ are equal to $\mathbf{r}_1$.
		Up to rearranging and switching rows (and their corresponding columns) we can apply the above argument to every row of $S$.
		Therefore each row of $S$ has either $t$ or $t+1$ entries equal to $\pm 3$.
		Moreover, by rearranging the rows (and corresponding columns), we can assume that the entries $\pm 3$ appear in diagonal blocks of $M$.
		Since $S$ has order $4t+1$, it must have precisely four distinct rows $\mathbf{r}_1$, $\mathbf{r}_2$, $\mathbf{r}_3$, and $\mathbf{r}_4$, three of which occur $t$ times and one occurring $t+1$ times.
		Without loss of generality, we assume that $\mathbf{r}_1$ occurs $t+1$ times.
		Observe that $\mathbf{r}_1$ corresponds to equality in \eqref{eqn:UB4a}, and hence $\mathbf{r}_1$ is equal to $\mathbf{v}$.
		 
		Since $M$ has rank $3$, each row $\mathbf{r}_i$ must be a linear combination of the other three.
		Write $\mathbf{r}_1 = c_2\mathbf{r}_2 +c_3\mathbf{r}_3 + c_4\mathbf{r}_4$.
		Up to switching equivalence, we can assume that all entries of $\mathbf{r}_1$ are positive.
		Equating coefficients gives us the following equations
		\begin{align*}
			3 &= c_2 + c_3 + c_4 \\
			1 &= 3 \varepsilon_2 c_2 + \sigma c_3 + \tau c_4 \\
			1 &= \sigma c_2 + 3 \varepsilon_3 c_3 + \xi c_4 \\
			1 &= \tau c_2 + \xi c_3 + 3 \varepsilon_4 c_4,
		\end{align*}
		where $\varepsilon_2, \varepsilon_3, \varepsilon_4, \sigma, \tau,$ and $\xi$ are members of $\{\pm 1\}$.
		It is straightforward to check that the only solution to this system of equations is when $\varepsilon_2 = \varepsilon_3 = \varepsilon_4 = 1$, $\sigma = \tau = \xi = -1$, and $c_2 = c_3 = c_4 = 1$.
		This gives us the desired form for $M$.
		\end{proof}
		
		Before completing the proof of Theorem~\ref{thm:skewMain} we need an intermediate result.
		
		\begin{lemma}\label{lem:bfofSM}
			Let $S$ be a skew-symmetric Seidel matrix of order $4t+1$ with spectrum 
			\[
				\{ [\pm \sqrt{-4t-3}]^{2t-1}, [\pm \sqrt{-3}]^1, [0]^1\}.
			\]
			Let $M=(4t+3)I+S^2$ and let $P \in \mathcal O_{4t+1}(\mathbb{Z})$ satisfy Equation~\eqref{eqn:sqForm3}.
			Then
			\[
				P^\top SM P = \begin{pmatrix}
					O & O & O & O \\
					O & O & -4J_t & 4 J_t \\
					O & 4 J_t & O & -4 J_t \\
					O & -4 J_t & 4 J_t & O
				\end{pmatrix}.
			\]
		\end{lemma}
		\begin{proof}
			We write $\mathbf{r}_i$ for the $i$th distinct row of $P^\top MP$.
			Consider the matrix $X:=P^\top SM P$.
			First, since $X$ is the sum of odd powers of a skew-symmetric matrix, it is clear that $X$ has zero diagonal.
			Since $\mathbf{r}_1$ is in the kernel of $P^\top SP$ and $X$ is skew-symmetric, the first $t+1$ rows and columns of $X$ are all zero.
			Denote by $s_j$ the $j$th row of $P^\top SP$.
			Suppose $j > t+1$.
			Think of $s_j$ as consisting of $4$ parts, each one corresponding to a distinct row of $P^\top MP$.
			Let $p_l$ denote the number of positive entries in the $l$th part of $s_j$.
			Since $s_j \mathbf{r}_1^\top = 0$, we have
			\begin{equation}
				\label{eqn:ps}
				3 p_1 + p_2 + p_3 + p_4 = 3t +1.
			\end{equation}
			Using equation \eqref{eqn:ps}, one can obtain the following modulo $8$ congruence
			\[
				s_j \cdot \mathbf{r}_t \equiv \begin{cases}
					0, & \text{ if the $j$th row of $P^\top MP$ is $\mathbf{r}_t$ } \\
					4, & \text{ otherwise }
				\end{cases} \pmod 8.
			\]
		
			The matrix $\sqrt{-1}X$ has spectrum $\{ [0]^{4t-1}, [\pm 4t\sqrt{3}]^1 \}$.
			By interlacing, $\sqrt{-1}X$ does not contain a principal submatrix of the form 
			\[
				\begin{pmatrix}
					O & \xi \sqrt{-1}J_t \\
					\xi \sqrt{-1}J_t & O
				\end{pmatrix},
			\]
			where $|\xi| \geqslant 8$.
			Hence $s_j \cdot \mathbf{r}_t \in \{0,\pm 4\}$.
			Observe that $\mathbf{r}_1 = \mathbf{r}_2 + \mathbf{r}_3 + \mathbf{r}_4$.
			Hence, up to relabelling the $\mathbf{r}_i$, the matrix $X$ has the required form.
		\end{proof}
		
		Finally, we can complete the proof of Theorem~\ref{thm:skewMain}.
		
		\begin{theorem}\label{thm:finalPart}
			Let $S$ be a skew-symmetric Seidel matrix of order $4t+1$ with spectrum 
			\[
				\{ [\pm \sqrt{-4t-3}]^{2t-1}, [\pm \sqrt{-3}]^1, [0]^1\}.
			\]
			Then there exists a vector $\mathbf{x} \in \{\pm 1\}^{4t+1}$ such that the matrix
			\[
				\begin{pmatrix}
					S & \mathbf{x} \\
					-\mathbf{x}^\top & 0
				\end{pmatrix}
			\]
			is a skew-symmetric Seidel matrix with spectrum $\{ [\pm \sqrt{-4t-3}]^{2t}, [\pm \sqrt{-1}]^1 \}$.
		\end{theorem}
		\begin{proof}
			Let $M=(4t+3)I+S^2$ and let $P \in \mathcal O_{4t+1}(\mathbb{Z})$ satisfy Equation~\eqref{eqn:sqForm3}.
			We write $\mathbf{r}_i$ for the $i$th distinct row of $P^\top M P$.
			By Lemma~\ref{lem:bfofSM} we have
			\begin{align*}
				P^\top S P \mathbf{r}_2 &= \mathbf{r}_3 - \mathbf{r}_4; \\
				P^\top S P \mathbf{r}_3 &= \mathbf{r}_4 - \mathbf{r}_2; \\
				P^\top S P \mathbf{r}_4 &= \mathbf{r}_2 - \mathbf{r}_3.
			\end{align*}
			Set $\mathbf{x} = (\mathbf{r}_1 -\mathbf{r}_2)/2$.
			Observe that $\mathbf{x}$ is a $\{\pm 1\}$-vector.
			Let $S^\prime$ be the $(4t+2) \times (4t+2)$ matrix given by
			\[
				S^\prime = \begin{pmatrix}
					P^\top S P & \mathbf{x} \\
					-\mathbf{x}^\top & 0
				\end{pmatrix}.
			\]
			
			Let $\mathbf{v}$ be a $(\pm\sqrt{-4t-3})$-eigenvector of $P^\top S P$.
			Then $\mathbf{v}$ is in the null space of $M$ and hence $\mathbf{v}$ is orthogonal to each of the vectors $\mathbf{r}_i$ ($1 \leqslant i \leqslant 4$).
			Thus $\mathbf{v}$ is orthogonal to $\mathbf{x}$.
			It follows that by appending a zero to $\mathbf{v}$ one obtains a $(\pm\sqrt{-4t-3})$-eigenvector of $S^\prime$.
			
			Let $\mathbf{w} = (\mathbf{r}_3 - \mathbf{r}_4)/4 + \sqrt{-1} (\mathbf{r}_1 + \mathbf{r}_2)/4$.
			One can check that by appending $-1$ to the vectors $\mathbf{w}$ and $\overline{\mathbf{w}}$, we obtain eigenvectors of $S^\prime$ corresponding to the eigenvalues $\pm \sqrt{-1}$.
		
			Therefore $S^\prime$ has eigenvalues $\pm \sqrt{-4t-3}$ each with multiplicity at least ${2t-1}$ and $\pm \sqrt{-1}$ each with multiplicity at least $1$.
			Let $\lambda$ and $\mu$ denote the remaining two unknown eigenvalues.
			By Proposition~\ref{pro:basic}, we deduce that $\lambda = - \mu = \sqrt{-4t-3}$.
		\end{proof}
		
		\begin{remark}
			In the proof we can apply a similar argument taking $(\mathbf{r}_1 - \mathbf{r}_3)/2$ or $(\mathbf{r}_1 - \mathbf{r}_4)/2$ as the extra row and column for $S^\prime$.
		\end{remark}

\section*{Acknowledgements} 
\label{sec:acknowledgements}

We thank the referees for their careful reading of the manuscript and their helpful comments.

\bibliographystyle{myplain}

\end{document}